\documentclass[preprint,12pt]{amsart}

\usepackage[utf8]{inputenc}
\usepackage[left=4 cm, right= 4 cm]{geometry}
\usepackage{lmodern}
\usepackage[english,activeacute]{babel}
\usepackage{empheq}
\usepackage{amsmath}
\usepackage{amsthm}
\usepackage{amssymb}
\usepackage{amsrefs}
\usepackage{enumitem}
\usepackage{graphics}
\usepackage[all]{xy}
\usepackage{hyperref}
\usepackage{bookmark}
\usepackage[toc,page]{appendix}

\newtheorem{theorem}{Theorem}[section]
\newtheorem{corollary}[theorem]{Corollary}
\newtheorem{lemma}[theorem]{Lemma}
\newtheorem{prop}[theorem]{Proposition}

\newtheorem{comentario}[theorem]{Remark}

\theoremstyle{definition}
\newtheorem{definicion}[theorem]{Definition}

\numberwithin{equation}{section}

\newenvironment{myproofprop}[1][Proof of the proposition]{\begin{proof}[#1]}{\end{proof}}

\newcommand{\R}{\mathbb{R}}
\newcommand{\C}{\mathbb{C}}

\newcommand{\N}{\mathbb{N}}
\newcommand{\Z}{\mathbb{Z}}
\newcommand{\SL}{\mathrm{SL}}
\newcommand{\hs}{\mathbb{H}}

\newcommand{\spec}{\mathrm{Spec}}
\newcommand{\proj}{\mathrm{Proj}}
\newcommand{\mP}{\mathbb{P}}

\makeatother

\title{Degenerations of non-simple abelian surfaces}
\author{Nelson Alvarado}

\address{N. Alvarado \\Dipartimento di Matematica, Università degli studi di Roma Tor Vergata , Via della ricerca scientifica, 00133 Roma, Italy}
\email{alvarado@mat.uniroma2.it}

\begin{document}
\maketitle

\begin{abstract}
We study degenerations of non-simple principally polarized abelian surfaces to the boundary in the toroidal compactification of $\mathcal{A}_2$, and describe the degenerate abelian surfaces as well as the degenerate elliptic curves that live inside them. 
\end{abstract}

\section{Introduction}
Given a family of non-simple principally polarized abelian varieties, it is natural to ask how such a family could degenerate and try to understand the limit of the abelian subvarieties as well. For example, in \cite{ABH}, the authors studied degenerations of Prym varieties in the second Voronoi compactification of the moduli space of principally polarized abelian varieties. We recall \cite{Alexeev} that the second Voronoi compactification of $\mathcal{A}_g$ is functorial in the sense that the boundary elements can be interpreted as degenerate abelian varieties. In \cite{ABH}, the authors studied how Jacobian varieties of smooth projective curves that are double covers of another curve degenerate to the boundary in this toroidal compactification, all the while keeping track of how the abelian subvariety corresponding to the associated Prym variety degenerates as well. 

It seems interesting to understand how abelian subvarieties degenerate when the abelian variety they live in degenerates to the boundary of the moduli space. It is this question that we begin to look at in this article, particularly in the case of principally polarized abelian \textit{surfaces}.  

For $g=2$, all known toroidal compactifications of $\mathcal{A}_2$ coincide, and so we will be speaking of \textit{the} toroidal compactification of $\mathcal{A}_2$, and denote it by $\mathcal{A}_2^*$. Our article is separated into the study of two subproblems:
\begin{itemize}
    \item Understand the closure of the moduli space of non-simple principally polarized abelian surfaces in $\mathcal{A}_2^*$
    \item Study families of non-simple principally polarized abelian surfaces that degenerate to the boundary, and understand the degenerations of the elliptic curves. 
\end{itemize}

As for the first point, we give a complete characterization of the boundary points of the moduli space of non-simple principally polarized abelian surfaces. Indeed, our first main result says the following:

\begin{theorem}
Let $m\geq 2$, let $\mathcal{E}_m\subseteq\mathcal{A}_2$ be the moduli space of non-simple principally polarized abelian surfaces that contain an elliptic curve of degree $m$, and let $\mathcal{E}_{m}^{\ast}$ be the closure of $\mathcal{E}_{m}$ in $\mathcal{A}_{2}^{\ast}.$ Then set-theoretically we have that
$$\mathcal{E}_{m}^{\ast} = \mathcal{E}_{m}\sqcup K^{0}(1)[m]\sqcup\mP_{\infty}^{1},$$
where $K^{0}(1)[m]$ is a certain subvariety of the universal Kummer surface over $\mathcal{A}_1$, and $\mP_{\infty}^{1}$ is a projective line that lies on the border. Moreover, $\overline{K^0(1)[m]}\cap\mathbb{P}_\infty^1$ consists of $\varphi(m)+1$ points and the boundary of $\mathcal{E}_m^*\cap\mathcal{E}_n^*$ is $\mathbb{P}_\infty^1$ if $m\neq n$.
\end{theorem}

We remark that the intersection of the closure of $K^{0}(1)[m]$ with $\mP_\infty^{1}$ appears to be quite complicated with nasty singularities. It may be interesting to study in the future. 

As for the second part of the article, we take the boundary points studied in the above theorem and look at degenerate abelian surfaces lying above them using a construction of Mumford \cite{complete-deg} (see \cite{Hulek} as well). Among other properties, we obtain the following result:

\begin{theorem}
Using Mumford's construction, the pairs that consist of a degenerate abelian surface $X$ along with the degenerate subelliptic curve $F$ lying over the boundary of $\mathcal{E}_m^*$ are the following:
\begin{enumerate}
    \item $X$ is a $\mathbb{P}^1$-bundle over an elliptic curve and $E$ is an $m$-gon of $\mathbb{P}^1$s. 
    \item $X=\mathbb{P}^1\times\mathbb{P}^1$ and $F$ is a nodal curve if $m=2$
    \item $X=\mathbb{P}^1\times\mathbb{P}^1$ and $F$ is an $(m-1)$-gon of $\mathbb{P}^1$s if $m\geq 3$
    \item $X$ is the union of two copies of $\mathbb{P}^2$ with the blow-up of $\mathbb{P}^2$ at three points, and $F$ is an $(2m-1)$-gon of $\mathbb{P}^1$s.
\end{enumerate}
\end{theorem}

The structure of this article is as follows: In Section \ref{preliminaries} we go over the necessary preliminaries for non-simple principally polarized abelian surfaces, as well as the toroidal compactification of $\mathcal{A}_2$. In Section \ref{compactification} we describe the compactification of $\mathcal{E}_m$ in $\mathcal{A}_2^*$, and in Section \ref{degeneration} we look at the actual families that are being degenerated to the boundary. 

\textbf{Acknowledgments}

I would like to thank to my advisor Dr. Robert Auffarth, who suggested me this problem, for his many helpful comments regarding both the content and redaction of this article. 

This work was funded by a scolarship of the Mathematics department of Universidad de Chile and by the MIUR Excellence Department Project
MATH@TOV,
awarded to the Department of Mathematics, University of Rome, Tor
Vergata,
CUP E83C18000100006.

\section{Preliminaries}\label{preliminaries}

In this section we will review some preliminaries on non-simple principally polarized abelian surfaces and toroidal compactifications.

\subsection{Non-simple principally polarized abelian surfaces} Let $\mathcal{A}_2$ denote the (coarse) moduli space of principally polarized abelian surfaces, and define 
\[\mathcal{E}_m:=\{(A,\Theta)\in\mathcal{A}_2:A\text{ contains en elliptic curve }E\text{ with }\deg(\Theta|_E)=m\}.\]

The space $\mathcal{A}_2$ can be described as a quotient of the Siegel upper-half space $\mathbb{H}_2$ by the symplectic group $\mathrm{Sp}(4,\mathbb{Z})$; let
\[p:\mathbb{H}_2\to\mathcal{A}_2\]
denote the natural projection.

It is a now classical fact (see, for instance \cite[Corollary 5.5]{Kani}) that $\mathcal{E}_m$ is irreducible of dimension 2, and a matrix $(\tau_{ij})\in\mathbb{H}_2$ lies in $p^{-1}(\mathcal{E}_m)$ if and only if there exists a primitive vector $(a,b,c,d,e)\in\mathbb{Z}^5$ such that the following two equations are satisfied:
\begin{eqnarray}\label{1}m^2&=&b^2-4(ac+de)\\
\label{2}0&=&a\tau_{11}+b\tau_{12}+c\tau_{22}+d(\tau_{11}\tau_{22}-\tau_{12}^2)+e
\end{eqnarray}

Because of this, following \cite{num-char}, if $v:=(a,b,c,d,e)\in\mathbb{Z}^5$ is a vector that satisfies equation \eqref{1}, then we define
\[\mathbb{H}_2(v):=\{\tau\in\mathbb{H}_2:\tau\text{ satisfies }\eqref{2}\}.\]
Therefore 
\[\mathbb{E}_m:=p^{-1}(\mathcal{E}_m)=\bigcup_{v}\mathbb{H}_2(v)\]
where the union goes over all $v$ that satisfy equation \eqref{1}. By \cite[Lemma 3.7]{num-char}, we have that for every primitive $v\in\mathbb{Z}^5$ that satisfies equation \eqref{1},
\[\mathcal{E}_m=p(\mathbb{H}_2(v)).\]

\subsection{Toroidal compactifications} In this subsection we will briefly outline the necessary notations and terms we will use in what follows with regards to toroidal compactifications. For a full treatment of the toroidal compactification of $\mathcal{A}_2$ we recommend the reader consult \cite{Hulek}. In particular the notations we will use in this article come from this book. For a more general treatment we recommend the reader consult \cite{mumfordetal} or \cite{Namikawa}.

Throughout this subsection $F$ denotes a rational boundary component of $\hs_{2}$ (see \cite[Definition 3.5]{Hulek},and the remark below it). Starting from a collection 
$$\Sigma = \{\Sigma(F) : F \hspace{0.2cm}\text{rational boundary component}\}$$
of fans in the vector space $\mathrm{Sym}(2,\R)$ of $2\times 2$ real symmetric matrices satisfying certain \emph{admisibility condition} (see \cite[Definition 3.66]{Hulek}) we can obtain a so called toroidal compactification $\mathcal{A}_{2}^{\Sigma}.$ Throughout this work we write $\mathcal{A}_{2}^{\ast}$ for $\mathcal{A}_{2}^{\Sigma}$ when 
$\Sigma$ is the Legendre decomposition defined in \cite[Part I,Definition 3.117]{Hulek}.       

More concretely, the toroidal compactification $\mathcal{A}_{2}^{\ast}$ is obtained by gluing \emph{partial compactifications} $Y_{\Sigma(F)}(F)$ in a certain way, where $F$ runs over the set of rational boundary components of $\hs_{2}$ and $\Sigma(F)$ is a fan in $\mathrm{Sym}(2,\R)$ as in the previous paragraph. 

Each partial compactification is obtained by the following procedure:  
\begin{enumerate}
\item Take the partial quotient $X(F):= P^{\prime}(F)\setminus\hs_{2},$ where $P^{\prime}(F)$ is certain subgroup of $\mathrm{Sp}(4,\Z)$ defined in \cite[Definition 3.48]{Hulek}. 
\item Consider certain space $\mathfrak{X}_{\Sigma(F)}$ defined in \cite[Definition 3.52]{Hulek} which is obtained using the complex toric variety associated to the fan $\Sigma(F)$ and compute the closure $X_{\Sigma(F)}$ of $X(F)$ in $\mathfrak{X}_{\Sigma(F)}.$
\item Finally, take the quotient $Y_{\Sigma(F)}$ of $X_{\Sigma(F)}$ by the action of the group $P^{\prime\prime}(F)$ defined in \cite[Definition 3.48]{Hulek}. We will write 
\begin{equation}
\label{defqF}
q_{F}: X_{\Sigma(F)}\to Y_{\Sigma(F)} 
\end{equation}
for the corresponding quotient map.
\end{enumerate}

As we said before, $\mathcal{A}_{2}^{\ast}$ is obtained gluing the partial compactifications $Y_{\Sigma(F)}.$ That is, $\mathcal{A}_{2}^{\ast}$ is a quotient of the disjoint union of the partial compactifications under certain equivalence relation (\cite[Definition 3.74]{Hulek}). For a rational boundary component $G$ we will write 
\begin{equation}
\label{defpF}
p_{G}^{\ast}: Y_{\Sigma(G)}\rightarrow\mathcal{A}_{2}^{\ast}
\end{equation}
for the composition
$$Y_{\Sigma(G)}\hookrightarrow\coprod_{F} Y_{\Sigma(F)}\twoheadrightarrow\mathcal{A}_{2}^{\ast},$$
where the right arrow is the corresponding quotient map. 

Now, actually, to obtain $\mathcal{A}_{2}^{\ast}$ it is enough to study the partial compactifications corresponding to three distinguished rational boundary components: $F_{0} = \hs_{2},F_{1}\cong\hs$ and $F_{2} = \{\mathrm{id}\}.$ The component $F_k$ will be called the \textbf{corank $k$ boundary component}. It happens that the partial compactification $Y_{\Sigma(F_0)}$ is just $\mathcal{A}_{2},$ but the components $F_{1}$ and $F_{2}$ add new boundary points.

\begin{definicion}
We write 
$$\partial_{F_1}\mathcal{A}_{2}^{\ast} := p_{F_1}^{\ast}\left(Y_{\Sigma(F_1)}\right)-\mathcal{A}_{2} = p_{F_1}\left(P^{\prime\prime}(F_1)\setminus\partial X_{\Sigma(F_1)}\right),$$
where $\partial X_{\Sigma(F_1)} = X_{\Sigma(F_1)}-X(F_1)$ is the usual topological boundary.
We also write
$$\partial_{F_2}\mathcal{A}_{2}^{\ast} = p_{F_2}^{\ast}\left(Y_{\Sigma(F_2)}\right)-\left(\mathcal{A}_{2}^{\ast}\cup\partial_{F_1}\mathcal{A}_{2}^{\ast}\right).$$
\end{definicion}

We may think of $\partial_{F_k}$ as ``the set which is added to the boundary of $\mathcal{A}_{2}^{\ast}$ as a contribution of the corank $k$ boundary component''.

\subsubsection{Corank 1 boundary component}

For the corank 1 boundary component we have that the group $P^{\prime}(F_1)$ is isomorphic to $\Z$ and the quotient $X(F_1)$ can be identified with the image of the map $e_{1}:\hs_{2}\to\C^{\times}\times\C\times\hs$ given by 
$$\begin{pmatrix}
    \tau_{1} & \tau_{2} \\
    \tau_{2} & \tau_{3} \\
   \end{pmatrix}\longrightarrow (e^{2\pi i\tau_1},\tau_{2},\tau_{3}).$$  
The space $\mathfrak{X}_{\Sigma(F_1)}$ is nothing but $\C\times\C\times\hs$ and $\partial X_{\Sigma(F_1)} = \{0\}\times\C\times\hs.$ 

\begin{definicion}
\label{Kcerouno}
Let $K^{0}(1)$ be the the surface defined by the quotient of $\C\times\hs$ by the equivalence relation which identifies $(z,\tau)$ with  
$$\left(\frac{\varepsilon z +m\tau+n}{c\tau+d},\frac{a\tau+b}{c\tau+d}\right),$$
for $\varepsilon\in\{\pm 1\},m,n\in\Z$ and $\begin{pmatrix} 
a & b \\
c & d \\
\end{pmatrix}\in\SL(2,\Z).$
\end{definicion}

Note that we have a fibration $K^{0}(1)\to\mathcal{A}_{1}$ whose fiber over $[\tau]\in\mathcal{A}_{1}$ is $E_{\tau}/\langle z\mapsto -z\rangle,$ whenever $\mathrm{Stab}_{\SL(2,\Z)}(\tau)=\{\pm\mathrm{id}\}.$

Studying the action of the group $P^{\prime\prime}(F_1)$ on $\{0\}\times\C\times\hs = \partial X_{\Sigma(F_1)}$ (as in \cite[Part I, Proposition 3.101]{Hulek}) one can easily see that the quotient is precisely the surface $K^{0}(1)$ defined above. 

\subsubsection{Corank 2 boundary component}

The case of corank 2 is a little subtler. In this case the partial quotient $X(F_2)$ can be identified with the image of the map $e_{2}:\hs_{2}\to(\C^{\times})^{3}$ given by 
$$\begin{pmatrix}
    \tau_{1} & \tau_{2} \\
    \tau_{2} & \tau_{3} \\
   \end{pmatrix}\mapsto (e^{2\pi i\tau_1},e^{2\pi i\tau_{2}},e^{2\pi i\tau_{3}})$$
and the space $\mathfrak{X}_{\Sigma(F_2)}$ is the toroidal embedding $(\C^{\times})^{3}\subseteq T_{\Sigma(F_2)}$ associated to the Legendre decomposition defined in \cite[Definition 3.117]{Hulek}. 

The toroidal embedding $(\C^{\times})^{3}\subseteq T_{\Sigma(F_2)}$ is locally described by immersions $\iota_{n}: (\C^{\times})^{3}\hookrightarrow T_{n}\cong \C^{3}$ given by 
$$(t_{1},t_{2},t_{3})\mapsto(t_{1}t_{2}^{-(2n+1)}t_{3}^{n(n+1)},t_{2}t_{3}^{-n},t_{2}^{-1}t_{3}^{n+1}),$$
where $n\in\Z.$
By \cite[Remarks 3.149 and 3.156]{Hulek} we have that $\partial_{F_2}$ is a certain quotient of the axis $\{0\}\times\C\times\{0\}\subseteq T_{-1}$ and is isomorphic to $\mP^{1}.$ 

Summarizing, we have the following:

\begin{theorem}
\label{Adosestrella}
Let $\mathcal{A}_{2}^{\ast}$ be the toroidal compactification of $\mathcal{A}_{2}^{\ast}$ associated to the Legendre decomposition. Set-theoretically we have that
$$\mathcal{A}_{2}^{\ast} = \mathcal{A}_{2}\sqcup K^{0}(1)\sqcup\mP_{\infty}^{1},$$
where: 
\begin{itemize}
\item $K^{0}(1)$ is the relative Kummer surface $K^{0}(1)\to\mathcal{A}_{1}$ defined in Definition \ref{Kcerouno}, which is a certain quotient of $\{0\}\times\C\times\hs\subseteq\mathfrak{X}_{\Sigma(F_1)}$
\item $\mP_{\infty}^1$ is a copy of $\mP^{1}$ which is a quotient of the axis  $\{0\}\times\C\times\{0\}\subseteq\C^{3}\cong T_{-1}\subseteq\mathfrak{X}_{\Sigma(F_2)}.$
\end{itemize}
\end{theorem}

\begin{proof}

See \cite[Part I, Chapter 3, Theorem 3.151]{Hulek} and the remarks below it.

\end{proof}

\begin{definicion}
The $\mP_{\infty}^{1}$ from the above theorem will be called the \emph{peripheral $\mP^1$}.
\end{definicion}

\section{Compactification of $\mathcal{E}_{m}$}\label{compactification}

In this section we will determine the closure $\mathcal{E}_{m}^{\ast}$ of $\mathcal{E}_m$ in the toroidal compactification $\mathcal{A}_2^*$.

Again let $\mathbb{E}_{m} = p^{-1}(\mathcal{E}_{m}),$ where $p:\mathbb{H}_{2}\to\mathcal{A}_{2}$ is the canonical projection. In order to find the closure of $\mathcal{E}_m$ in $\mathcal{A}_2^*$ we have to compute the closure of the image of $\mathbb{E}_{m}$ in $X_{\Sigma(F_r)}$ for $r\in\{1,2\}.$ We denote by $e_{r}:\mathbb{H}_{2}\to X(F_r) = P^{\prime}(F_r)\setminus\mathbb{H}_{2}$ the partial quotient of corank $r.$

Using the notation of the previous section, we have that
$$\mathbb{E}_{m} = \bigcup_{v} \mathbb{H}_{2}(v),$$
where the union runs over the elements $v\in\Z^{5}$ which satisfy equation \eqref{1}. In what follows, when we use the notation $\mathbb{H}_2(v)$ we will \textbf{always} assume that $v$ satisfies equation \eqref{1}, and so we will frequently leave it out. 

In particular, we have that
$$\overline{e_{r}(\mathbb{E}_{m})}\supseteq\bigcup_{v}\overline{e_{r}(\mathbb{H}_{2}(v))},$$ 
where the closure is taken in $X_{\Sigma(F_r)}$ (which can naturally be seen as a subset of $\C^{3}$). What makes our calculations easier is the following result:

\begin{lemma}
We have that
$$\overline{e_{r}(\mathbb{E}_{m})} = \bigcup_{v}\overline{e_{r}(\mathbb{H}_{2}(v))}.$$ 
\end{lemma}

\begin{proof}
As before, let $Y_\Sigma(F)$ denote the quotient $P''(F)\backslash X_{\Sigma(F)}$, let $q_F:X_{\Sigma}(F)\to Y_\Sigma(F)$ be the natural projection and let $p_F^*:Y_\Sigma(F)\to\mathcal{A}_2^*$ be the gluing map. 

Take $x\in\overline{q_{F}e_{F}(\mathbb{E}_m)}$, and let $z\in q_F^{-1}(x)$. We affirm that $z\in\overline{e_F(\mathbb{H}_2(v))}$ for some $v\in\mathbb{Z}^5$ that satisfies equation \eqref{1}.

Since $P''(F)$ acts properly discontinuously in $X_{\Sigma(F)}(F)$, there exists an open neighborhood $U\subseteq X_{\Sigma(F)}(F)$ of $z$ such that $h(U)\cap U\neq\varnothing$ if and only if $h\in\mathrm{Stab}_{P''(F)}(z)$. Note that this stabilizer is finite. On the other hand, $Y_{\Sigma(F)}(F)$ is an analytic space and by \cite[Theorem II.4.7]{Dem}, $q_F(U)\cap q_Fe_F(\mathbb{E}_m)$ has a finite number of irreducible components. Therefore $x$ must be in the closure of one of these components. Now each of these components must be an irreducible component of $q_F(U)\cap q_Fe_F(\mathbb{H}_2(v))$ for some $v$, and we therefore have that $x\in\overline{q_Fe_F(\mathbb{H}_2(v))}$. 

We have an open embedding $\mathrm{Stab}(z)\backslash U\hookrightarrow Y_{\Sigma(F)}(F)$, and since $x\in\overline{q_Fe_F(\mathbb{H}_2(v))}$, we have that the image of $z$ in $\mathrm{Stab}(z)\backslash U$ is in the closure of the image of $e_F(\mathbb{H}_2(v))$ in $\mathrm{Stab}(z)\backslash U$. We therefore have that
\[z\in\overline{\bigcup_{h\in\mathrm{Stab}(z)}he_F(\mathbb{H}_2(v))}=\bigcup_{h\in\mathrm{Stab}(z)}\overline{he_F(\mathbb{H}_2(v))}\]
since $\mathrm{Stab}(z)$ is finite. Now $he_F(\mathbb{H}_2(v))=e_F(\mathbb{H}_2(w))$ for some $w$, and the proof is finished. 

\end{proof}

This lemma implies that in order to find the closure of $\mathcal{E}_m$ in $\mathcal{A}_2^*$, we only need to find the closure of $e_1(\mathbb{H}_2(v))$ for every $v$ that satisfies equation \eqref{1}.

\subsection{Corank 1 boundary component}

We will first study the closure in the corank 1 boundary component. As was said at section 2, in this case the partial quotient can be identified with the map $e_{1}:\mathbb{H}_{2}\to\C^{\times}\times\C\times\mathbb{H}$ given by 
$$\begin{pmatrix}
    \tau_{1} & \tau_{2} \\
    \tau_{2} & \tau_{3} \\
   \end{pmatrix}\longrightarrow (e^{2\pi i\tau_1},\tau_{2},\tau_{3}).$$  

\begin{lemma}
\label{tresdos}
Let $v=(a,b,c,d,e)\in\Z^{5}$ be a vector that satisfies \eqref{1}. We have that $\overline{e_{1}(\mathbb{H}_{2}(v))}$ intersects the boundary of $X_{\Sigma(F_1)}\subseteq\C\times\C\times\hs=\mathfrak{X}_{\Sigma(F_1)}$ if and only if $v=(0,\pm m,c,0,e)$ for some $c,e\in\Z$ with $\mathrm{gcd}(m,c,e)=1.$ More precisely, we have that
$$\overline{e_{1}(\mathbb{H}_{2}(v))}\cap X_{\Sigma(F_1)} = \begin{cases}
                                                                            \emptyset & \text{if $a=0$ or $d=0$} \\
                                                                            \left\{\left(0,\frac{c\tau+e}{m},\tau\right) : \tau\in\mathbb{H}\right\} & \text{if $v = (0,m,c,0,e)$} 
                                                                            \end{cases}.$$
\end{lemma}

\begin{proof}
We just have to study which points of the form $(0,z,\tau)\in\C\times\C\times\mathbb{H}$ are limits of sequences in $e_{1}(\mathbb{H}_{2}(v)).$ If $(0,z,\tau)$ is such a limit then there exists a sequence
\[\left\{\left(\begin{array}{cc}\omega_n&z_n\\z_n&\tau_n\end{array}\right)\right\}_{n\in\mathbb{N}}\subseteq\mathbb{H}_{2}\]
such that for all $n\in\mathbb{N}$,
\begin{equation}
\label{vn}
a\omega_{n}+bz_{n}+c\tau_{n}+d(z_{n}^{2}-\omega_{n}\tau_{n})+e = 0,
\end{equation}
and when $n\to\infty$,
\begin{itemize}
    \item $z_{n}\to z\in\C,$
    \item $\tau_{n}\to\tau\in\mathbb{H}$
    \item $\Im\omega_{n}\to\infty$.
\end{itemize}

The above conditions imply that the sequence $\{\widehat{\omega}_n\}_{n\in\N}$ given by 
$$\widehat{\omega}_{n} := (a -d\tau_{n})\omega_{n}$$
must have a finite limit, but it is easy to see that this is impossible if $a,d\neq 0.$ Now, condition \eqref{1} implies that necessarily $v = (0,\pm m, c,0,e)$, and since this vector is primitive, we have that $\mathrm{gcd}(m,c,e)=1$. If we replace \eqref{vn} with $a=d=0$ we get that $z_{n} = \mp\frac{c\tau_{n}+e}{m},$ so the limit is $\left(0,\mp\frac{c\tau+e}{m},\tau\right).$

\end{proof}

\begin{definicion}
\label{defKcerounom}
Let $K^{0}(1)[m]$ denote the subvariety of $K^{0}(1)$ (see Definition \ref{Kcerouno} given by the image in $K^{0}(1)$ of the set 
$$\left\{\left(\frac{c\tau+e}{m},\tau\right)\in\C\times\hs : \mathrm{g.c.d}(m,c,e) = 1\right\}.$$
\end{definicion}

Note that if $\mathrm{Stab}_{\SL(2,\Z)}(\tau) = \{\pm\mathrm{id}\}$ then the fiber of $K^{0}(1)[m]$ over $\tau$ is just the set of points of order $m$ in the elliptic curve $E_{\tau},$ mod $\pm 1 .$

The previous lemma therefore implies the following:

\begin{prop}
We have that
$$\mathcal{E}_{m}^{\ast}\cap\partial_{F_1}\mathcal{A}_{2}^{\ast} = K^{0}(1)[m].$$
\end{prop}

Surprisingly, this boundary component is actually irreducible: 

\begin{prop}
For all $m\in\mathbb{Z}$, $K^0(1)[m]$ is irreducible and if $m,n\in\mathbb{N}$ are different, then $K^0(1)[m]\cap K^0(1)[n]=\varnothing$.
\end{prop}
\begin{proof}
Define $X_{(c,e)}$ to be the image of 
\[\left\{\left(\frac{c\tau+e}{m},\tau\right)\: \tau\in\hs\right\}\]
in $\mathcal{A}_2^*$. We will show that $X_{(x,y)}=X_{(z,w)}$
if $\mathrm{g.c.d.}(x,y,m)=\mathrm{g.c.d.}(z,w,m)=1$.

First of all, we see that in order to get the corank 1 border component, we take the quotient of $\mathbb{C}\times\hs$ by the group
\[G:=\left\{\left(\begin{array}{ccc}\epsilon&m&n\\0&a&b\\0&c&d\end{array}\right):\begin{array}{c}\left(\begin{array}{cc}a&b\\c&d\end{array}\right)\in\mathrm{SL}(2,\mathbb{Z})\\m,n\in\mathbb{Z},\epsilon=\pm1\end{array}\right\}\]

Now the subgroup consisting of matrices with $\epsilon=1$ and $m,n=0$ acts as
\[\left(\begin{array}{ccc}1&0&0\\0&a&b\\0&c&d\end{array}\right):(z,\tau)\mapsto(z(c\tau+d)^{-1},(a\tau+b)(c\tau+d)^{-1}).\]
A brief calculation shows that the rational representation of the morphism
\[\mathbb{C}/\langle1,\tau\rangle\to\mathbb{C}/\langle1,(a\tau+b)(c\tau+d)^{-1}\rangle\]
\[z\mapsto z(c\tau+d)^{-1}\]
is the matrix 
\[\left(\begin{array}{cc}a&-b\\-c&d\end{array}\right),\] 
and so this same matrix modulo $m$ gives the representation of the action of the morphism on the torsion points of both tori.

Now if $(\overline{x},\overline{y}),(\overline{z},\overline{w})\in(\mathbb{Z}/m\mathbb{Z})^2$ are points of order $m$, it is an easy exercise to see that there exists an element $M\in\mathrm{SL}(2,m)$ such that $M(\overline{x},\overline{y})=(\overline{z},\overline{w})$. On the other hand, the reduction map $\pi:\mathrm{SL}(2,\mathbb{Z})\to\mathrm{SL}(2,m)$ is surjective (which can be seen by observing that both groups are generated by elementary matrices, and each elementary matrix modulo $m$ is in the image of the reduction map), and so there exists
\[N=\left(\begin{array}{cc}r&s\\u&v\end{array}\right)\in\mathrm{SL}(2,\mathbb{Z})\]
such that $\pi(N)=M$. Now we get that for any $\tau$,
\[\left(\begin{array}{ccc}1&0&0\\0&r&-s\\0&-u&v\end{array}\right)\cdot\left(\frac{x+\tau y}{m},\tau\right)\equiv \left(\frac{z+\tau' w}{m},\tau'\right)\text{ (mod }G) \]
where $\tau'=(r\tau-s)(u\tau-v)^{-1}$. In particular, $X_{(x,y)}= X_{(z,w)}$, and so we are done. 

As for the statement that $K^0(1)[m]\cap K^0(1)[n]=\varnothing$ if $m\neq n$, this is clear since a point of order $m$ on a torus cannot be of order $n$. 
\end{proof}

\subsection{Corank 2 boundary component}

Now we study the closure in the corank 2 boundary component. In this case the partial quotient can be identified with the map $e_{2}:\mathbb{H}_{2}\to(\C^{\times})^{3}$ given by 
$$\begin{pmatrix}
    \tau_{1} & \tau_{2} \\
    \tau_{2} & \tau_{3} \\
   \end{pmatrix}\mapsto (e^{2\pi i\tau_1},e^{2\pi i\tau_{2}},e^{2\pi i\tau_{3}}).$$
and we have a holomorphic map $\psi:\mathbb{H}_{2}\to\C^{3}$ given by $\iota\circ e_{2},$ where $\iota:(\C^{\times})^{3}\to\C^{3}$ is the map given by $\iota(t_{1},t_{2},t_{3}) = (t_{1}t_{2}^{-1},t_{2},t_{3}t_{2}^{-1})$ which allows us to see $X(F_2):=P^{\prime}(F_2)\setminus\hs_{2}$ as a subset of the toroidal embedding $(\C^{\times})^{3}\hookrightarrow T_{\Sigma(F_2)}.$ According to \cite[Part I, Lemma 3.137 and Proposition 3.144]{Hulek} we have that $\mP_{\infty}^{1}$ is certain quotient of the axis $\{0\}\times\C\times\{0\}.$

\begin{lemma} 
\label{trescuatro}
The axis $\{0\}\times\C\times\{0\}$ lies in the closure of $\psi(\mathbb{E}_{m})\subseteq\C^3.$
\end{lemma}

\begin{proof}
Consider the period matrices defined by 
$$T_{z,\tau} = \begin{pmatrix}
       \tau & z \\
        z & \frac{1}{m-1}\left(\tau-(m-2)z\right) \\
    \end{pmatrix},$$
where $z$ is a fixed complex number and $\tau$ is such that $T_{z,r}\in\mathbb{H}_2$. We have that the above matrix is an element of $\mathbb{H}_{2}(v)$ for $v=(1,-(m-2),-(m-1),0,0)$ and when $\Im (\tau)\to\infty$ we obtain that 
\[\psi(T_{z,\tau})\to (0,e^{2\pi i z},0).\] 
Since $z\in\C$ is arbitrary we get that $\{0\}\times\C\times\{0\}$ is contained in $\overline{\psi(\mathbb{E}_{m})}$ as we wanted to see. 

\end{proof}

\begin{corollary}
We have that 
$$\mathcal{E}_{m}^{\ast}\cap\partial_{F_2}\mathcal{A}_{2}^{\ast} = \mP_{\infty}^{1},$$
where $\mP_{\infty}^{1}$ is the \emph{peripheral $\mP^{1}$} introduced in Theorem \ref{Adosestrella}
\end{corollary}

By \cite[Remark 3.156]{Hulek}, the closure of the corank 1 boundary component of $\mathcal{A}_2$ in $\mathcal{A}_2^*$ is isomorphic to the Kummer modular surface $\mathcal{K}_1\to\mathbb{P}^1$, and therefore $\mathbb{P}^1_\infty$ is the fiber over $\infty$ of this family.

We now need to see how the boundary components intersect each other.

\begin{definicion}
For $m\in\mathbb{Z}$, let $C_m$ be the closure of $K^0(1)[m]$ in $\mathcal{A}_2^*$.
\end{definicion}

\begin{prop}
The curve $C_m$ intersects $\mathbb{P}_\infty^1$ at $\varphi(m)+1$ points where $\varphi$ is the Euler totient function.
\end{prop}
\begin{proof}
For $a,b\in\{0,1,\ldots,m-1\}$, the image of $\{(\frac{1}{m}(a+b\tau),\tau):\tau\in\mathbb{H}\}$ under the partial quotient map $e:\C\times\mathbb{H}\to\C^\times\times\C^\times$ is
$$\{(e^{2\pi i(a+b\tau)/m},e^{2\pi i\tau}):\tau\in \mathbb{H}\}=\{(t^b\rho^a,t^m):0<|t|<1\},$$
where $t=e^{2\pi i\tau/m}$ and $\rho=e^{2\pi i/m}$ is a primitive $m$th root of unity. The image of this curve in the torus embedding $\C^\times\times\C^\times\hookrightarrow\C\times\C$ where $(u,v)\mapsto(uv^{-n},u^{-1}v^{n+1})$ is the curve
$$\{(t^{b-nm}\rho^a,t^{(n+1)m-b}\rho^{-a}):0<|t|<1\}.$$ 
We see that the closure of this curve in $\C\times\C$ contains boundary points if and only if 
$$-1+\frac{b}{m}\leq n\leq \frac{b}{m}.$$
Since $b<m$, we obtain that $n\in\{-1,0\}$. Under identification by the action of $\mbox{Sp}(4,\Z)$, it is easy to see that we only need to consider the case $n=0$, and so we have the curve
$$\{(t^b\rho^a,t^{m-b}\rho^{-a}):0<|t|<1\}.$$
We need to analyze the behavior of the image of this curve in the quotient $\C^2/((x,y)\sim(y,x))$.

If $b=0$, then the boundary point that we add is the image of $(\rho^a,0)$ in the quotient. We see that the involution $(x,y)\mapsto(y,x)$ does not fix these points, they are not equivalent for different values of $b$, and the curve is smooth at this point. Here there are $\varphi(m)$ possible values for $a$. 

If $b\neq0$, the point we add is (the image of) $(0,0)$, which shows that there is only one other point of intersection. 

As a remark, we see that the quotient map $\C^2\to\C^2/((x,y)\sim(y,x))$ can be identified with $(x,y)\mapsto(x+y,xy)$, and so we need to look at the curve
$$\{(t^a\rho^b+t^{m-a}\rho^{-b},t^m):0<|t|<1\}.$$
Since $t^a\rho^b+t^{m-a}\rho^{-b}=t^a(\rho^b+t^{m-2a}\rho^{-b})$, we notice that if $a\mid m$ then the image is smooth, but if $a\nmid m$ then it is singular. 
\end{proof}

\begin{comentario}
The deepest degeneration point on the intersection of $C_m$ with $\mathbb{P}^1_\infty$ seems to be quite complicated and is of high multiplicity; moreover, some of the branches of this intersection are singular and some are smooth. 
\end{comentario}

Putting the above results together we directly obtain the following:

\begin{theorem}
\label{teoremaclausura}
Let $m\geq 2$ and $\mathcal{E}_{m}^{\ast}$ be the closure of $\mathcal{E}_{m}$ in $\mathcal{A}_{2}.$ 

Set-theoretically we have that
$$\mathcal{E}_{m}^{\ast} = \mathcal{E}_{m}\sqcup K^{0}(1)[m]\sqcup\mP_{\infty}^{1},$$
where $K^{0}(1)[m]$ is the subvariety of $K^{0}(1)$ from Definition \ref{defKcerounom} and $\mP_{\infty}^{1}$ is the \emph{peripheral $\mP^{1}$} introduced in Theorem \ref{Adosestrella}. Moreover, $\overline{K^0(1)[m]}\cap\mathbb{P}_\infty^1$ consists of $\varphi(m)+1$ points and the boundary of $\mathcal{E}_m^*\cap\mathcal{E}_n^*$ is $\mathbb{P}_\infty^1$ if $m\neq n$.
\end{theorem}

\section{Degenerate abelian surfaces}\label{degeneration}

Here we study Mumford's final example in \cite{complete-deg} in the global way that is worked out in \cite[Part II, Chapter 2]{Hulek} in order to give an interpretation of the boundary points of $\mathcal{E}_{m}^{\ast}$ as degenerate abelian surfaces which contain a degenerate elliptic curve. We will only gloss over the construction; the interested reader is invited to look at \cite[Part II, Chapters 1 and 2]{Hulek} for more details.

\subsection{Mumford's construction}

First we need to fix some notation. Let $A=\C[T_{1},T_{2},T_{3}],$ $S = \spec\hspace{0.1cm}A\cong\C^{3}$ and consider the $A$-scheme
$$\widetilde{G} = \spec\left(\frac{A[U,U^{-1},V,V^{-1},W,W^{-1}]}{\left<UVW-1\right>}\right)\cong\C^{3}\times(\C^{\times})^{2}.$$ 
Now, let $K$ be the quotient field of $A$ and denote by $\mathbb{Y}$ the subgroup of $\widetilde{G}(K)$ generated by $r$, $s$ and $t$, where
\begin{align*}
    r & = (T_{2}T_{3},T_{3}^{-1},T_{2}^{-1}) \\
    s & = (T_{3}^{-1},T_{1}T_{3},T_{1}^{-1}) \\
    t & = (T_{2}^{-1},T_{1}^{-1},T_{1}T_{2})
\end{align*}
(note that this group is clearly isomorphic to $\Z^{3}/\Z(1,1,1)\cong\Z^{2}$).  

These objects are related to the study of degenerations of abelian surfaces as follows:
\begin{itemize}
    \item Consider the action of $\Z^{2}$ in $\mathbb{H}_{2}\times(\C^{\times})^{2}$ given by
$$(m,n) : \left(\tau,(w_{1},w_{2})\right)\rightarrow\left(\tau,(e^{2\pi mi\tau_{11}}e^{2\pi ni\tau_{12}}w_{1},e^{2\pi mi\tau_{12}}e^{2\pi ni \tau_{22}}w_{2})\right)$$
and denote 
$$\widehat{\mathcal{S}}:= \left(\mathbb{H}_{2}\times(\C^{\times})^{2}\right)/\Z^{2}.$$
\item The action of $P^{\prime}(F_2)$ on $\mathbb{H}_{2}$ induces an action on $\widehat{\mathcal{S}}.$ We denote
$$\mathcal{S} = P^{\prime}(F_2)\setminus\widehat{\mathcal{S}}.$$
\item The projection $\mathbb{H}_{2}\times(\C^{\times})^{2}\rightarrow\mathbb{H}_{2}$ induces a map $\mathcal{S}\rightarrow X(F_2)$ which is the right vertical map in the following commutative diagram:
\begin{equation*}
\xymatrix{ \mathbb{H}_{2}\times(\C^{\times})^{2}\ar[d]\ar[rr]^{/\Z^{2}} & & \widehat{\mathcal{S}}\ar[d]\ar[rr]^{/P^{\prime}(F_2)} & & \mathcal{S}\ar[d] \\
\mathbb{H}_{2}\ar@{=}[rr] & & \mathbb{H}_{2} \ar[rr]_{/P^{\prime}(F_2)} & & X(F_2) }
\end{equation*} 
\end{itemize}

 Clearly we have that the fiber $\mathcal{S}_{[\tau]}$ of $\mathcal{S}$ over $[\tau]$ is the abelian surface $\C^{2}/(\Z^{2}\oplus\tau\Z^{2}).$

On the other hand, we have that $\mathbb{Y}$ acts on $\widetilde{G}(K)$ by multiplication and hence for every $T\in(\C^{\times})^{3}$ we have an action of $\mathbb{Y}$ on the fiber $\widetilde{G}_{T}.$  Consider the map $\Psi: \mathbb{H}_{2}\to\C^{3}$ given by the composition $j\circ e_{2},$ where $j:(\C^{\times})^{3}\to\C^{3}$ is the map given by
\begin{equation}
\label{varphi}
(t_{1},t_{2},t_{3})\mapsto(t_{1}t_{2},t_{2}t_{3},t_{2}^{-1}).
\end{equation}
It is easy to see that if $T=\Psi(\tau)$ for some $\tau\in\mathbb{H}_{2}$ then the quotient $\widetilde{G}_{T}/\mathbb{Y}$ is also isomorphic to the abelian surface $\C^{2}/(\Z^{2}\oplus\tau\Z^{2}).$ Moreover, we have the following commutative square: 
\begin{equation*}
\xymatrix{ \mathcal{S}\ar[r]^{[\Psi\times\mathrm{id}]}\ar[d] & \widetilde{G}/\mathbb{Y}\ar[d] \\
X(F_2)\ar[r]_{j} & \C^{3} }
\end{equation*}
which induces isomorphisms 
$$\mathcal{S}_{\tau}\cong\widetilde{G}_{\Psi(\tau)}/\mathbb{Y}.$$

The idea of Mumford's construction is to obtain a family $\mathcal{P}\to U$ defined over an open subset $U\subseteq\C^{3}$ containing $X(F_2)$ as a dense subset in such a way that the fibers over $X(F_2)$ are abelian surfaces. For this purpose, let $R = R_{\Phi,\Sigma}$ be the graded $A$-algebra defined in \cite[Part II, Definition 1.10]{Hulek} and $\widetilde{P} = \proj\hspace{0.1cm}R_{\Phi,\Sigma}.$ We have that $\widetilde{G}$ is a dense open subscheme of $\widetilde{P}$ and the inclusion $\widetilde{G}\hookrightarrow\widetilde{P}$ is a morphism of $A$-schemes. Denote by $U$ the analytic open subset of $\C^3$ given by the interior of the closure of $\Psi(\mathbb{H}_{2})$ and define
\begin{equation}
\label{PU}
\widetilde{P}_{U} := \widetilde{P}\times_{\C^3} U.
\end{equation}
We then have a natural analytic morphism $\widetilde{P}_{U}\to U$
and, furthermore, we have the following result: 

\begin{theorem}
\label{definicionP}
The group $\mathbb{Y}$ acts properly discontinuosly on $\widetilde{P},$ we have a morphism $\widetilde{P}_{U}/\mathbb{Y}\to U$ and this morphism defines a proper and flat family over $U$ where a general fiber is an abelian surface.   
\end{theorem}

\begin{proof}
See \cite[Part II, Theorem 3.10]{Hulek}. For the definition of the action see \cite[Part II, Definition 1.9]{Hulek}
\end{proof}

\begin{definicion}
We write $P\to U$ for the family of the above theorem.
\end{definicion}

\subsection{Families of non-simple abelian surfaces}

Using the matrices employed in the proofs of Lemmas \ref{tresdos} and \ref{trescuatro} we can obtain families whose fibers are non-simple principally polarized abelian surfaces which contain an elliptic curve of exponent $m$. For this, consider the sets
\begin{equation}
\label{Oce}
O_{(c,e)} = \Psi\left(\left\{\begin{pmatrix}
                     \mu & -\frac{c\tau+e}{m} \\
                     -\frac{c\tau+e}{m} & \tau
                     \end{pmatrix}\in\mathbb{H}_{2}\right\}\right)\subseteq\C^{3},
\end{equation}
and 
\begin{equation}
\label{Oinfty}
O_{\infty} = \Psi\left(\left\{\begin{pmatrix}
                                   \tau & z \\
                                   z & \frac{1}{m-1}\left(\tau-(m-2)z\right) 
                                   \end{pmatrix}\in\mathbb{H}_{2}\right\}\right)\subseteq\C^{3}.
\end{equation}
Restricting $\mathcal{S}\to X(F_2)$ to $O_{i},$ $i\in\{(c,e)\in\N^{2} : \mathrm{g.c.d}(m,c,e)=1\}\cup\{\infty\}$ we get families $\mathcal{S}_{i}\rightarrow O_{i}$ with the desired properties.

Since the fiber $(\mathcal{S}_{i})_{T}$ is a non-simple abelian surface for every $T\in O_{i}$, we have that all of these surfaces contain an elliptic curve. Now we will show that we can actually put these elliptic curves together to form a family. More precisely, we will prove the following:

\begin{prop}
\label{existenciafamilias}
Let $i\in\{(c,e)\in\N^{2}: \mathrm{g.c.d}(m,c,e)=1\}\cup\{\infty\}$ and consider the morphism $\mathcal{S}_{i}\to O_{i}$ defined above. There exists a morphism $\mathcal{K}_{i}\to O_{i}$ and an $O_{i}$-morphism $\mathcal{K}_{i}\to\mathcal{S}_{i}$ such that
\begin{enumerate}
    \item each fiber of the map $\mathcal{K}_{i}\to O_{i}$ is an elliptic curve
    \item for every $T\in O_{i}$ the induced map $(\mathcal{K}_{i})_{T}\hookrightarrow(\mathcal{S}_{i})_{T}$ is an inclusion whose image is an abelian subvariety of exponent $m$
\end{enumerate}
\end{prop}

In other words, we now proceed to show that there exists a commutative diagram
\begin{equation*}
    \xymatrix{ \mathcal{K}_{i} \ar[rr]\ar[rd] &  & \mathcal{S}_{i}\ar[ld] \\
    & O_{i} & }
\end{equation*}
with the properties stated in the proposition above.

 Let $\Lambda_{\tau}$ denote the lattice in $\C^{2}$ generated by the columns of the matrix $\begin{pmatrix} \mathrm{id} & \tau \end{pmatrix}$ and write $A_{\tau}: = \C^{2}/\Lambda_{\tau}.$ We have that every elliptic curve $E\leq A_{\tau}$ is of the form $W/(W\cap\Lambda_{\tau}),$ where $W$ is a linear subspace of $\C^{2}$ with $\dim_{\C}(W)=1.$ According to \cite{num-char} we have that if $\tau\in\mathbb{H}_{2}(v)$ then we can explicitly compute a linear subspace $W_{v}\leq\C^{2}$ in such a way that the curve $W_{v}/(W_{v}\cap\Lambda_{\tau})$ is a subvariety of exponent $m$ of $A_{\tau}$: 

\begin{lemma} We have the following:
\begin{enumerate}[label=\roman*)]
\item Let $(c,e)\in\Z^{2}$ with $\mathrm{g.c.d}(m,c,e)=1$ as in Lemma \ref{tresdos} and $v_{(c,e)}=(0,m,c,0,e).$ For $\tau\in\mathbb{H}_{2}(v_{(c,e)})$ we have that the linear subpspace $W_{(c,e)} := \left<(1,0)\right>_{\C}\leq\C^2$ defines an elliptic curve of exponent $m$ contained in $A_{\tau}.$ 
\item Let $v_{\infty} = (1,-(m-2),-(m-1),0,0)$ and $\tau\in\mathbb{H}_{2}(v_{\infty}).$ We have that the linear subspace $W_{\infty} := \left<(-1,1)\right>_{\C}$ defines an elliptic curve of exponent $m$ contained in $A_{\tau}.$ 
\end{enumerate}
\end{lemma}

\begin{proof}
Let $J$ denote the matrix 
$$\begin{pmatrix}
   0 & -\mathrm{id} \\
   \mathrm{id} & 0 
   \end{pmatrix}\in\mathbb{M}_{4\times 4}(\R),$$
where $0\in\mathbb{M}_{2\times 2}(\R)$ is the zero matrix and for $v=(a,b,c,d,e)\in\Z^{5}$ define
$$M_{v} = \begin{pmatrix}
0 & d & -\frac{b-m}{2} & a \\
-d & 0 & -c & \frac{b+m}{2} \\
\frac{b-m}{2} & c & 0 & -e \\
-a & -\frac{b+m}{2} & e & 0
\end{pmatrix}\in\mathbb{M}_{4\times 4}(\R).$$
Now, let $C$ denote the canonical basis of $\R^{4}$ and $B$ denote the $\R$-basis of $\C^{2}$ given by the columns of the matrix $\begin{pmatrix} \mathrm{id} & \tau \end{pmatrix}.$ According to \cite[Proposition 3.2]{num-char} we have that if $\tau\in\mathbb{H}_{2}(v)$ then the image of the $\R$-linear transformation $\R^{4}\to\C^{2}$ given in basis $C,B$ by the matrix $m\cdot\mathrm{id}-JM_{v}$ defines an abelian subvariety of $A_{\tau}$ of dimension 1 and exponent $m.$ 

Using the above fact, a direct computation finishes the proof. 

\end{proof}

\begin{proof}[Proof of Proposition \ref{existenciafamilias}]
The images in $(\C^{\times})^2$ of the linear spaces $W_{(c,e)}$ and $W_{\infty}$ of the Lemma above by the exponential map are the sets
$$L_{(c,e)} = \{(V,U)\in(\C^{\times})^{2} : U = 1\}$$
and 
$$L_{\infty} = \{(V,U)\in(\C^{\times})^2 : UV = 1\}.$$
We get the desired families $\mathcal{K}_{i}\to O_{i}$ simply by restricting the map $\mathcal{S}_{i}\to O_{i}$ to the image of the set $O_{i}\times L_{i}\subseteq\C^{3}\times(\C^{\times})^{2}$ in $\mathcal{S}.$

\end{proof}

In the next section we will use Mumford's construction to extend the above families to proper and flat families defined over  $\overline{O}_{i}\cap U$, obtaining degenerate abelian surfaces which contain a degenerate elliptic curve.

\subsection{Degenerations of abelian subvarieties}

Recall that an analytic space $A_{0}$ is called a \textbf{degenerate abelian surface} (resp. degenerate elliptic curve) if there exists a smooth space $B,$ a dense open set $U\subseteq B$ and a proper and flat family $\mathcal{X}\to B$ such that the fiber $\mathcal{X}_{t}$ is an abelian surface (resp. elliptic curve) for every $t\in U$ and there exists $t_{0}\in B$ such that $\mathcal{X}_{t_0}\cong A_{0}.$

In this subsection we will prove the following result:

\begin{theorem}
\label{main}
Let $m\geq 2$ be a fixed integer, $\mathcal{E}_{m}$ the subset of $\mathcal{A}_{2}$ given by the classes of non-simple principally polarized abelian surfaces which contain an elliptic curve of exponent $m$ and $\mathcal{E}_{m}^{\ast}$ be the closure of $\mathcal{E}_{m}$ in $\mathcal{A}_{2}^{\ast},$ where $\mathcal{A}_{2}^{\ast}$ is the toroidal compactification of $\mathcal{A}_{2}$ associated to the Legendre decomposition. There exists a finite collection of finite maps $\pi_{i}:X_{i}\to\mathcal{E}_{m}^{\ast}$ such that
\begin{enumerate}
    \item The boundary of $\mathcal{E}_{m}^{\ast}$ is contained in the union of the images of the maps $\pi_{i}$
    \item For each $i$ there are (explicit) analytic morphisms $Q_{i}\to X_{i},P_{i}\to X_{i}$ and a $X_{i}$-morphism $Q_{i}\to P_{i}$ such that:
     \begin{enumerate}
         \item For every $T\in X_{i}$ we have that the fiber $(Q_{i})_{T}$ is a degenerate elliptic curve and the fiber $(P_{i})_{T}$ is a degenerate abelian surface. Furthermore, if $\pi_{i}(T)\in\mathcal{E}_{m}$ then $(Q_{i})_{T}$ is an elliptic curve and $(P_{i})_{T}$ is a non-simple principally polarized abelian surface
         \item For every $T\in X_i$ the morphism $Q_{i}\to P_{i}$ induces an inclusion $(Q_{i})_{T}\hookrightarrow (P_{i})_{T}.$ Furthermore, if $\pi_{i}(T)\in\mathcal{E}_{m}$ then this inclusion represents a subvariety of exponent $m.$
     \end{enumerate}
\end{enumerate}
\end{theorem}

Once the above result is established, we will prove the following theorem:

\begin{theorem}\label{main2}
The pairs (normalised degenerate abelian surface, degenerate elliptic curve) which arise under the families $P_{i},Q_{i}$ are
\begin{itemize}
    \item (certain $\mP^{1}$-bundle over certain elliptic curve, $m$-gon of $\mP^{1}$'s) 
    \item ($\mP^{1}\times\mP^{1}$, nodal curve) if $m=2$
    \item ($\mP^{1}\times\mP^{1}$, $(m-1)$-gon of $\mP^{1}$'s) if $m\geq 3$
    \item (the union of two copies of $\mP^{2}$ with one of the blow-up of $\mP^{2}$ at three points, $(2m-1)$-gon of $\mP^{1}$'s)
\end{itemize}
where each pair can be explicitly calculated. 
\end{theorem}

In other words, every point in the boundary of $\mathcal{E}_{m}^{\ast}$ can be interpreted as a degenerate abelian surface which contains a degenerate elliptic curve. 

We will proceed in several steps. 

\begin{definicion}
Let $(c,e)\in\Z^2$ with $\mathrm{g.c.d}(c,e,m)=1$ and $0\leq c,e\leq m.$

For $i=\infty$ or $i=(c,e)$ as above, denote  
\begin{equation}
\label{defVi}
X_{i} = \overline{O_i}\cap U\subseteq\C^3,
\end{equation}
where $U$ is the interior of the closure of $\Psi(\hs_2)$ as in \eqref{PU} and $O_{i}$ is the open set defined in \eqref{Oce} for $i=(c,e)$ and in \eqref{Oinfty} for $i=\infty.$
\end{definicion}

\begin{definicion}
\label{mapspi}
Define the map $\pi_{i}:X_{i}\to\mathcal{E}_{m}^{\ast}$ as the composition $p_{F_2}^{\ast}\circ q_{F_2},$ where $p_{F_2}^{\ast}$ and $q_{F_2}$ are the maps introduced in equations \eqref{defpF} and \eqref{defqF}, respectively. 
\end{definicion} 

Directly from the definitions and our proof of Theorem \ref{teoremaclausura} we obtain the following:

\begin{prop}
The maps $\pi_{i}:X_{i}\to\mathcal{E}_{m}^{\ast}$ from Definition \ref{mapspi} are finite and the sets $\pi_{i}(X_i)$ cover the border of $\mathcal{E}_{m}^{\ast}$ 
\end{prop}

\begin{definicion}
We define the analytic space $P_{i}$ as the restriction of $P$ (from Theorem \ref{definicionP}) to $X_{i}.$ That is, take 
$$P_{i} := \left.P\right|_{X_i} = P\times_{U} X_{i}.$$
We write $P_{i}\to X_{i}$ for the natural morphism arising from the universal property of fiber product.
\end{definicion}

From \cite[Part II, Theorem 3.10]{Hulek} we have the following

\begin{prop}
The map $P_{i}\to X_{i}$ defined above is proper, flat and its fibers are the ones stated at Theorem \ref{main}.
\end{prop}

Now, motivated by our proof of the Proposition \ref{existenciafamilias}, we consider the following schemes and spaces:

\begin{definicion}
Let 
$$\widetilde{H} = \spec\left(\frac{A[U,U^{-1},V,V^{-1},W,W^{-1}]}{\left<UVW-1,U-1\right>}\right),$$
and
$$\widetilde{H}_{\infty} = \spec\left(\frac{A[U,U^{-1},V,V^{-1},W,W^{-1}]}{\left<UVW-1,W-1\right>}\right),$$
which are subschemes of $\widetilde{G}.$ 
\end{definicion}

\begin{definicion}
We denote by $H_{(c,e)}$ (resp. $H_{\infty}$) the restriction of $\widetilde{H}$ (resp $\widetilde{H}_{\infty}$) to $O_{(c,e)}$ (resp. $O_{\infty}).$ 
\end{definicion}

Using the notation from the previous definition we have the following commutative diagram:
\begin{equation*}
\xymatrix{ \mathcal{K}_{i}\ar[rr]^{[\Psi\times\mathrm{id}]}\ar[dr] & & H_{i}/\mathbb{Y}\ar[ld] \\
 & O_{i} &  }
\end{equation*}
where $i=(c,e)$ or $i=\infty.$

\begin{prop}
\label{podemosextender}
The map $[\Psi\times\mathrm{id}]$ from the above diagram induces isomorphisms
$$\left(\mathcal{K}_{i}\right)_{T}\cong (H_{i})_{T}/\mathbb{Y}$$
for every $T\in O_{i}.$
\end{prop}

To prove this proposition we will first need an important lemma:

\begin{lemma}
\label{subgrupoactuando}
Consider the subgroups of $\mathbb{Y}$ given by $\mathbb{Y}_{c}:= (r^{c}s^{m})$ and $\mathbb{Y}_{\infty} = (r^{-(m-1)}s).$ We have that $\mathbb{Y}_{c}$ (resp. $\mathbb{Y}_{\infty}$) acts on $H_{(c,e)}$ (resp. $H_{\infty}$) and $H_{(c,e)}/\mathbb{Y} = H_{(c,e)}/\mathbb{Y}_{c}$ (resp. $H_{\infty}/\mathbb{Y} = H_{\infty}/\mathbb{Y}_{\infty}$). 
\end{lemma}

\begin{proof}
From the definition of the action of $\mathbb{Y}$ on $\widetilde{P}$ given in \cite[Part II, Definition 1.9]{Hulek} we have that 
$$S_{ar+bs}(U) = T_{2}^{-a}T_{3}^{b-a}U$$
and 
$$S_{ar+bs}(W) = T_{2}^{-a}T_{1}^{-b}W.$$
On the other hand, it is easy to see that 
\begin{equation}
\label{contencionesOic}
O_{(c,e)}\subseteq\{(T_{1},T_{2},T_{3})\in(\C^{\times})^{3} : T_{2}^{c} = T_{3}^{m-c}\} 
\end{equation}
and
\begin{equation}
\label{contencionOinfty}
O_{\infty}\subseteq \{(T_{1},T_{2},T_{3})\in(\C^{\times})^{3} : T_{1} = T_{2}^{m-1}\}, 
\end{equation}
thus an easy computation finishes the proof.
\end{proof}

\begin{myproofprop}
From the definition of the map $[\Psi\times\mathrm{id}]$ we have that $(H_{i})_{T}$ corresponds to the sets $L_{(c,e)}$ and $L_{\infty}$ from the proof of Proposition \ref{existenciafamilias}, which are both isomorphic to $\C^{\times}$. On the other hand, using the previous lemma it is easy to see that the action agrees with the one that defines the corresponding elliptic curve as a subquotient of $(\C^{\times})^{2}.$ 

\end{myproofprop}

Recall that we want to extend the families $\mathcal{K}_{i}\to O_{i}$ defined in Proposition \ref{existenciafamilias} to families defined over the sets $X_{i}$ defined in \ref{defVi} From the above proposition, we can do that using the scheme $\widetilde{P}$ from Mumford's construction.

\begin{definicion}
Let $\widetilde{Q}$ (resp $\widetilde{Q}_{\infty}$) be the closure of $\widetilde{H}$ (resp. $\widetilde{H}_{\infty}$) in $\widetilde{P}.$ 
\end{definicion}

We then have a morphism
$$\widetilde{Q}\to \spec\hspace{0.1cm}A\cong\C^{3}$$
(resp. $\widetilde{Q}_{\infty}\to\spec\hspace{0.1cm}A$) given by the composition of the inclusion $\widetilde{Q}\to\widetilde{P}$ (resp. $\widetilde{Q}_{\infty}\to\widetilde{P}$) with the natural map $\widetilde{P}\to\spec\hspace{0.1cm}A.$

\begin{definicion}
We define $Q_{(c,e)}$ (resp. $Q_{\infty}$) as the image in 
$$P_{i} = P\times_{U} X_{i} = \left[\left(\widetilde{P}\times_{\C^3} U\right)\times_{U}\times X_{i}\right]/\mathbb{Y}$$ of the restriction of $\widetilde{Q}$ (resp. $\widetilde{Q}_{\infty}$) to $X_{(c,e)}$ (resp. $X_{\infty}$). 
\end{definicion}

\begin{definicion}
\label{defQi}
We write $Q_{i}\to X_{i}$ and  $Q_{i}\to P_{i}$ for the natural maps arising from the definition of $Q_{i}.$ 
\end{definicion}

We have that the maps from the previous definition fit in the following commutative diagram
\begin{equation*}
    \xymatrix{H_{i}/\mathbb{Y} \ar[r]\ar[d] & Q_{i} \ar[r]\ar[d] & P_{i}\ar[d] \\
    O_{i}\ar@{^{(}->}[r] & X_{i} \ar@{=}[r] & X_{i} }
\end{equation*}

From Proposition \ref{podemosextender} it is then clear that the morphism $Q_{i}\to X_{i}$ extends the family $\mathcal{K}_{i}\to O_{i}$ from Proposition \ref{existenciafamilias}. We claim that this morphism is actually proper and flat. To prove this we will need to describe the fibers of this map. For this we only need to describe the fiber of the map $\widetilde{Q}\to\spec\hspace{0.1cm}A\cong\C^3$ (resp. $\widetilde{Q}_{\infty}\to\spec\hspace{0.1cm}A$) over a point $T\in X_{(c,e)}$ (resp. $T\in X_{\infty}$) and see which points are identified under the action of $\mathbb{Y}.$

\begin{lemma}
\label{lemaproj}
We have that 
$$\widetilde{Q} = \proj\left(R/I\right)$$
and 
$$\widetilde{Q}_{\infty} = \proj\left(R/I_{\infty}\right),$$
where $I$ (resp. $I_{\infty}$) is the homogenization of the ideal that defines $\widetilde{H}$ (resp. $\widetilde{H}_{\infty}$).
\end{lemma}

\begin{proof}
It is obvious that $\widetilde{Q}$ (resp. $\widetilde{Q}_{\infty}$) is an open subset of $\proj(R/I)$ (resp $\proj(R/I_{\infty}).$  Now, as both $I$ and $I_{\infty}$ are radical in $R,$ the result follows.

\end{proof}

Using the above, our problem becomes merely algebraic. 

\begin{prop}
\label{fibras-tilde}
Let $i=(c,e)$ or $i=\infty$ and let $\widetilde{Q},\widetilde{Q}_{\infty}$ and $X_{i}$ be as in Lemma \ref{lemaproj} and equation \eqref{defVi}, respectively. 
\begin{enumerate}
    \item Consider the groups $\mathbb{Y}_{c}$ and $\mathbb{Y}_{\infty},$ subgroups of $\mathbb{Y}$ defined in Lemma \ref{subgrupoactuando}. We have that for $T\in X_{(c,e)}$ (resp. $T\in X_{\infty})$ the only identifications in $\left(\widetilde{Q}\right)_{T}$ (resp. $\left(\widetilde{Q}_{\infty}\right)_T$) which arise from the action of $\mathbb{Y}$ come from $\mathbb{Y}_{c}$ (resp. $\mathbb{Y}_{\infty}$). 
    \item We have that the fiber of the map $\widetilde{Q}\to \spec\hspace{0.1cm}A$ (resp $\widetilde{Q}_{\infty}\to\spec\hspace{0.1cm}A$) over a point $T\in X_{(c,e)}$ (resp $T\in X_{\infty}$) is a chain of $\mP^{1}$'s.
\end{enumerate}
\end{prop}

\begin{proof}
\begin{enumerate}[leftmargin = 0.4cm]
    \item According to Lemma \ref{lemaproj} it is enough to find the elements of $\mathbb{Y}$ which preserve the equation that defines the ideal $I$ (resp $I_{\infty}$) given the condition that $T\in X_{(c,e)}$ (resp $T\in X_{\infty}$). These are all easy calculations recalling the contentions \eqref{contencionesOic} and \eqref{contencionOinfty}.
    \item Recall that for every $T$ we have that $\widetilde{Q}_{T}$ (resp $\left(\widetilde{Q}_{\infty}\right)_{T}$) is contained in $\widetilde{P}_{T}.$ For $i=(c,e)$ we have that the boundary points of $X_i$ lie in $\{0\}\times(\C^{\times})^{2}$ and according to \cite[Part II, Proposition 2.15]{Hulek} for those kinds of points the fiber $\widetilde{P}_{T}$ is a chain of $\C^{\times}\times\mP^{1}$'s. That is, over those points the fiber $\widetilde{P}_{T}$ has countably many irreducible components $Z_{k},k\in\Z,$ each one isomorphic to $\C^{\times}\times\mP^{1}$ and for $k\neq l$ these components satisfy that 
    $$Z_{k}\cap Z_{l} \cong \begin{cases}
                             \C^{\times} & \text{if $|k-l|=1$} \\
                             \emptyset & \text{in other case}
                             \end{cases},$$
    where $(u,[0:1])\in Z_{k}$ is identified with $(u,[1:0])\in Z_{k+1}.$ Moreover, each of those components has an explicit algebraic description (which can be found in the proof of \cite[Part II, Prop. 2.15]{Hulek}) which can be worked out. Using that explicit description and Lemma \ref{lemaproj} it is easy to see that
    $$\widetilde{Q}\cap Z_{k}\cong\mP^1$$
    and for $k\neq l$ we have that 
    $$\widetilde{Q}\cap Z_{k}\cap Z_{l}\cong\begin{cases}
    \{pt\} & \text{if $|k-l|=1$} \\
    \emptyset & \text{in any other case} 
    \end{cases},$$
    that is, $\widetilde{Q}_{T}$ is a chain of $\mP^{1}$'s, as we wanted to see.
    
    Now, for $i=\infty$ we have to distinguish between the cases $T\neq 0$ and $T=0.$ For $T\neq 0$ $\widetilde{P}_{T}$ is a net of quadrics, that is, it has countable many irreducible components $Z_{a,b},$ $a,b\in\Z,$ each one isomorphic to $\mP^{1}\times\mP^{1}$ with the intersections suggested by the terminology. This fiber can be visualized as in the figure \ref{P-tildeT} below, where each rectangle represents a copy of $\mP^{1}\times\mP^{1}$.  On the other hand, for $T=0$ the fiber $\widetilde{P}_{T}$ is illustrated by the figure \ref{P-tilde0}, where each triangle is a copy of $\mP^{2}$ and each hexagon represents the blow-up of $\mP^{2}$ at three general points. In both cases, using local coordinates and the action of the corresponding group, it can be worked out that the fiber is a chain of $\mP^{1}$'s and can be visualized as in Figures \ref{Qinfinitonocero} and \ref{Qinfinitocero}. 
    
\begin{figure}[h!]
\centering
\includegraphics[width=0.7\linewidth]{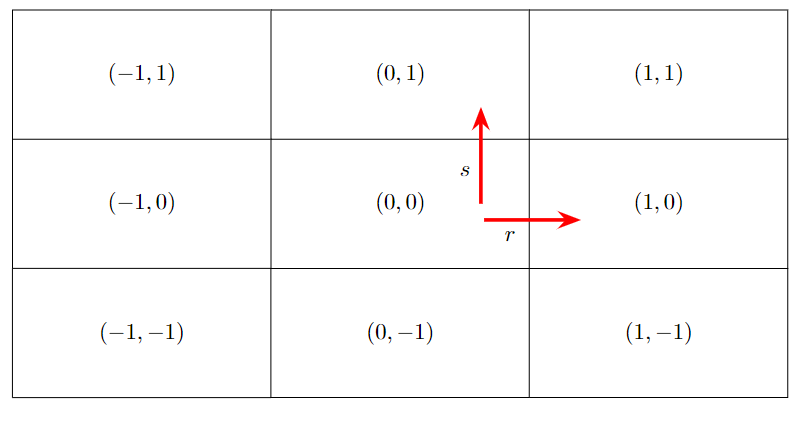}
\caption{Illustration of $\widetilde{P}_{T}$ for $T=(0,0,T_{3})$ with $T_{3}\neq 0$}
\label{P-tildeT}
\end{figure}

\begin{figure}[h!]
\centering
\includegraphics[width=0.7\linewidth]{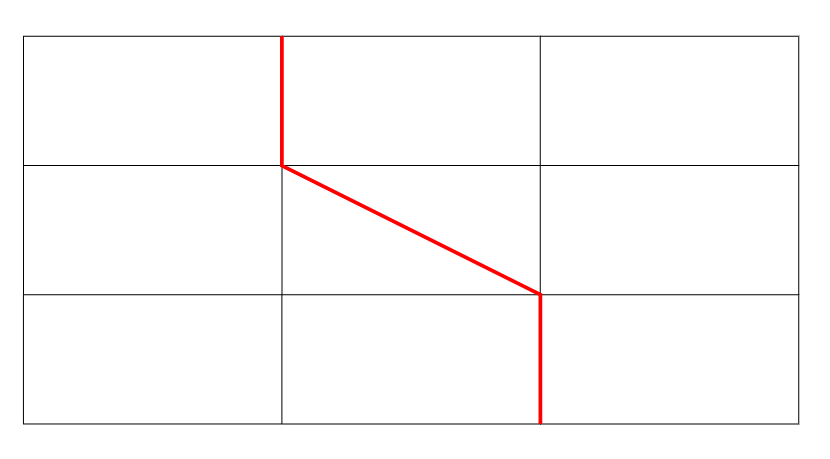}
\caption{Illustration of $\widetilde{Q}_{\infty}$ over $T=(0,0,T_{3})$ for $m=3$ and $T\neq 0$}
\label{Qinfinitonocero}
\end{figure}

    \begin{figure}[h]
\centering
\includegraphics[width=0.7\linewidth]{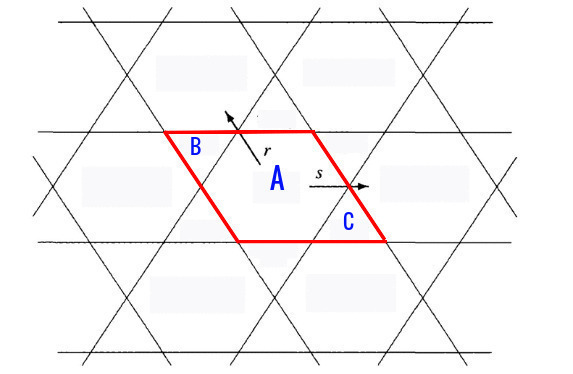}
\caption{Illustration of $\widetilde{P}_{0}.$ The whole fiber is $\mathbb{Y}$-translated from the parallelogram bounded in red and A,B,C are the irreducible components of the fundamental region}
\label{P-tilde0}
\end{figure}

    \begin{figure}[h!]
\centering
\includegraphics[width=0.7\linewidth]{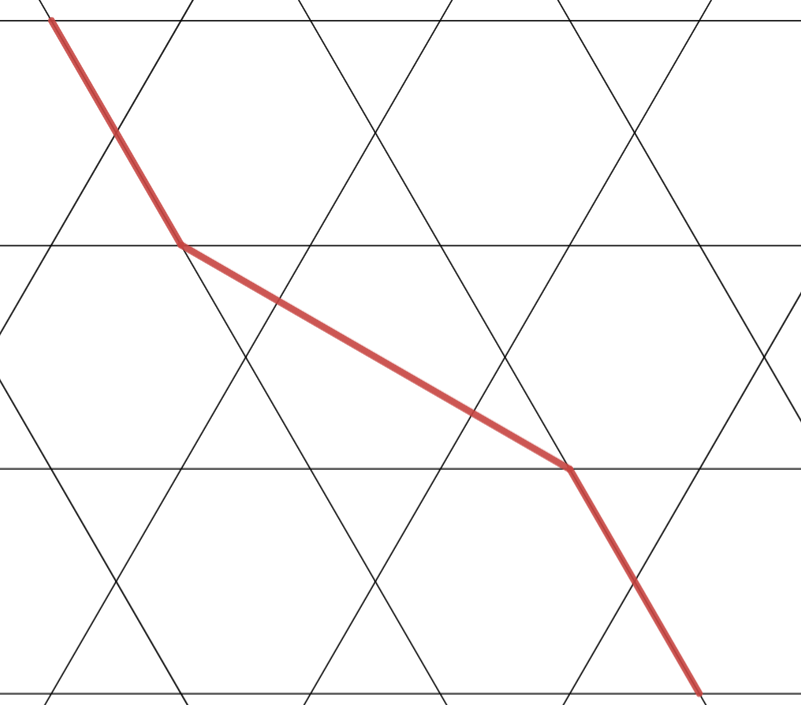}
\caption{Illustration of $\widetilde{Q}_{\infty}$ over $T=(0,0,0)$ for $m=3$}
\label{Qinfinitocero}
\end{figure}

\end{enumerate}
\end{proof}

Now, using the  previous result we can get the description of the fibers and show that those fibers are actually degenerate elliptic curves. 

 \begin{prop}
Let $m\geq 2$ be a fixed positive integer number and for $i=(c,e) $ or $i=\infty$ consider the morphisms $Q_{i}\to X_{i}$ defined in Definition \ref{defQi}. We have that 
\begin{enumerate}
    \item For $T\in(\C^{\times})^{3}$ the fiber of the map $Q_{i}\to X_{i}$ over $T$ is an elliptic curve
    \item For $T\notin(\C^{\times})^3$ the fiber of the map $Q_{(c,e)}\to X_{(c,e)}$ over $T$ is a $m$-gon of $\mP^{1}$'s
    \item For $T\notin(\C^{\times})^{3}$ the fiber of the map $Q_{\infty}\to X_{\infty}$ over $T$ is either 
    \begin{itemize}
        \item a nodal curve if $m=2$ and $T\neq 0$
        \item a $(m-1)$-gon of $\mP^{1}$'s if $m\geq 3$ and $T\neq 0$
        \item a $(2m-1)$-gon of $\mP^{1}$'s if $T=0$
    \end{itemize}
    \item The morphisms $Q_{i}\to X_{i}$ are proper and flat. In particular, each of the fibers above can be called a degenerate elliptic curve
\end{enumerate}
\end{prop}

\begin{proof}
\begin{enumerate}
 \item This is just a consequence from the construction of the morphisms and our proof of Theorem \ref{teoremaclausura}. 
   \item Using the action of the group $\mathbb{Y}_{(c,e)}$ it is just necessary to see which points in  $$\left(\widetilde{Q}_{c}\right)_{T}\cap \bigcup_{k=0}^{m-1} Z_{k}$$
   are identified by the generator $r^{c}s^{m},$ where $Z_{k}\cong\C^{\times}\times\mP^{1}$ are the irreducible component of the fibre $\widetilde{P}_{T}$ (see our proof of Proposition \ref{fibras-tilde}). This is easy because $r$ acts on each $Z_{k}$ and $s$ sends $Z_{k}$ into $Z_{k+1}.$
   \item As before, it is enough to see which points in 
   $$\left(\widetilde{Q}_{\infty}\right)_{T}\cap \bigcup_{k=0}^{m-2} Z_{k,0}$$
   are identified by the generator $r^{-(m-1)}s,$ where $Z_{a,b}\cong\mP^{1}\times\mP^{1}$ are the irreducible components of $\widetilde{P}_{T}$ (see Proposition \ref{fibras-tilde}). This is an easy computation using that $r^{-(m-1)}s$ maps $Z_{k,0}$ to $Z_{k-(m-1),1}$ and the description of the intersections given by \cite[Part II, Proposition 2.15]{Hulek}.
   \item Given the above description of the fibers, the result follows from \cite[Lemma 1, p.56]{complexgeom} and \cite[Corollary p.158]{complexgeom}.
\end{enumerate}

\end{proof}

Having established which are the degenerate elliptic curves that arises from our construction, Step 4 is complete and hence we have proved Theorem \ref{main} and Theorem \ref{main2} stated at the beginning of this section.

\end{document}